\documentclass[10pt,a4paper]{article}
\usepackage[utf8]{inputenc}

\usepackage{amsthm,amsmath,amsfonts,amssymb}
\usepackage{graphicx,url}

\renewcommand{\geq}{\geqslant}
\renewcommand{\leq}{\leqslant}
\renewcommand{\ge}{\geqslant}
\renewcommand{\le}{\leqslant}

\def\eref#1{$(\ref{#1})$}
\def\sref#1{\S$\ref{#1}$}
\def\lref#1{Lemma~$\ref{#1}$}
\def\tref#1{Theorem~$\ref{#1}$}
\def\Tref#1{Table~$\ref{#1}$}
\def\fref#1{Figure~$\ref{#1}$}
\def\cref#1{Corollary~$\ref{#1}$}
\def\dref#1{Definition~$\ref{#1}$}

\usepackage{multirow}
\usepackage[left=2cm,right=2cm,top=2cm,bottom=2cm]{geometry}

\title{On Perfect Sequence Covering Arrays}

\author{Aidan R. Gentle and Ian M. Wanless\\
\small School of Mathematics\\[-0.5ex]
\small Monash University\\[-0.5ex]
\small Vic 3800, Australia\\
\small\tt \{aidan.gentle, ian.wanless\}@monash.edu}

\date{}

\newtheorem{theorem}{Theorem}[section]
\newtheorem{lemma}[theorem]{Lemma}
\newtheorem{cor}[theorem]{Corollary}

\theoremstyle{definition}

\theoremstyle{definition}
\newtheorem{definition}[theorem]{Definition}

\newcommand{\sym}{\mathcal{S}}
\begin{document}
\maketitle

\begin{abstract}
  A PSCA$(v, t, \lambda)$ is a multiset of permutations of the
  $v$-element alphabet $\{0, \dots, v-1\}$ such that every sequence of
  $t$ distinct elements of the alphabet appears in the specified order
  in exactly $\lambda$ of the permutations. For $v \geq t \geq 2$, we
  define $g(v, t)$ to be the smallest positive integer $\lambda$ such
  that a PSCA$(v, t, \lambda)$ exists. We show that $g(6, 3) = g(7, 3)
  = g(7, 4) = 2$ and $g(8, 3) = 3$. Using suitable permutation
  representations of groups we make improvements to the upper bounds
  on $g(v, t)$ for many values of $v \leq 32$ and $3\le t\le 6$.  We
  also prove a number of restrictions on the distribution of symbols
  among the columns of a PSCA.
\end{abstract}

\section{Introduction}\label{s:intro}

For positive integers $v$ and $t$ with $v \geq t$, we let $[v] = \{ 0, \dots, v-1 \}$, $\sym_{v}$ be the set of permutations of $[v]$ and $\sym_{v, t}$ be the set of ordered sequences of $t$ distinct elements of $[v]$. For $\pi \in \sym_{v}$ and $s = (s_{0}, \dots, s_{t-1}) \in \sym_{v, t}$ we say that $s$ is \emph{covered} by $\pi$ if $\pi^{-1}(s_{i}) < \pi^{-1}(s_{i + 1})$ for $0 \leq i \leq t - 2$. A \emph{perfect sequence covering array} with order $v$, strength $t$ and multiplicity $\lambda$, denoted by PSCA$(v, t, \lambda)$, is a multiset $X$ of permutations in $\sym_{v}$ such that every sequence in $\sym_{v, t}$ is covered by exactly $\lambda$ permutations in $X$. If we let $T$ be a $t$-subset of $[v]$, then there are $t!$ orderings of the symbols of $T$, each of which must be covered by $\lambda$ permutations in a PSCA$(v, t, \lambda)$. Furthermore, every permutation in a PSCA$(v, t, \lambda)$ covers exactly one ordering of $T$, so a PSCA$(v, t, \lambda)$ must consist of $t!\lambda$ permutations.

Perfect sequence covering arrays were introduced by Yuster \cite{Yus19} in 2020 as a variant of sequence covering arrays. \emph{Sequence covering arrays,} denoted by SCA$(v, t)$, are sets of permutations in $\sym_{v}$ in which every sequence in $\sym_{v, t}$ is covered by at least one permutation in the set. The study of SCAs dates back to Spencer \cite{Spen71} in 1971. They are useful for constructing test suites for situations where the order of operations may be important.

For $v \geq t$ define $g(v, t)$ to be the smallest positive integer $\lambda$ such that a PSCA$(v, t, \lambda)$ exists. Observe that $\sym_{v}$ is a PSCA$(v, t, v!/t!)$ so $g(v, t)$ is well defined and $g(v, t) \leq v!/t!$. Note that if $v > t$ and we remove the symbol $v-1$ from every permutation of a PSCA$(v, t, \lambda)$, then we obtain a PSCA$(v-1, t, \lambda)$ and hence $g(v, t) \geq g(v-1, t)$. For $2 \leq t' \leq t$, a PSCA$(v, t, \lambda)$ is also a PSCA$(v, t', \lambda \binom{t}{t'})$ so $g(v, t') \leq \binom{t}{t'}g(v, t)$.

The question of when $g(v, t) = 1$ has received particular attention. Not only would a PSCA$(v, t, 1)$ be the smallest possible SCA$(v, t)$, but it is also an object of interest in coding theory. A \emph{$(v - t)$-deletion correcting code} is a set $X$ of permutations in $\sym_{v}$ such that every sequence in $\sym_{v, t}$ is covered by \textit{at most} one permutation in $X$. Hence, a PSCA$(v, t, 1)$ would be the largest possible $(v - t)$-deletion correcting code. For more on deletion correcting codes, see \cite{Kle04, Lev91}. Note that $\sym_{v}$ forms a PSCA$(v, v, 1)$. At the other end of the spectrum, if $t = 2$, then we can take any permutation in $\sym_{v}$ and its reverse to form a PSCA$(v, 2, 1)$. Therefore, $g(v, v) = g(v, 2) = 1$. Levenshtein \cite{Lev91} proved that $g(t + 1, t) = 1$ for $t \geq 3$. Mathon and van Trung \cite{Math99} proved that a PSCA(5, 3, 1) does not exist (we provide a new proof of this fact in \sref{s:distro}). As demonstrated above, $g(v, t) \geq g(v-1, t)$ so it follows that a PSCA$(v, 3, 1)$ does not exist for $v \geq 5$. Therefore, when $t = 3$, we have $g(v,t)>1$ for $v>t + 1$. It was initially conjectured by Levenshtein that this property would hold for any $t \geq 3$, however that was later shown to be false for $t = 4$ by Mathon and van Trung \cite{Math99}, who presented a PSCA(6, 4, 1). On the other hand, Mathon and van Trung computationally proved that neither a PSCA(7, 5, 1) nor a PSCA(8, 6, 1) exists, thus confirming Levenshtein's conjecture for $t \in  \{ 5, 6 \}$. They also found that a PSCA(7, 4, 1) does not exist. A combinatorial proof of this last fact was later given by Klein \cite{Kle04}. Chee et al. \cite{Chee13} proved that $g(2t, t) > 1$ for $t \geq 3$.

Yuster \cite{Yus19} proved that $g(5, 3) = 2$.
In \sref{s:comp}, we show that $g(6, 3) = g(7, 3) = g(7, 4) = 2$, $g(8, 3) = 3$ and $g(8, 4)\ge3$. 
We state here a definition which we expand upon further in \sref{s:distro}.

\begin{definition}\label{d:distro}
For a multiset $X$ of permutations of $[v]$, a symbol $w \in [v]$, and for $0 \leq i \leq v-1$, we define
\begin{align*}
d_{w}(i) := \big\vert \{ \pi \in X : \pi(i) = w \} \big\vert.
\end{align*}
We refer to the vector $\boldsymbol{d}_{w} = \big(d_{w}(0), \dots, d_{w}(v-1)\big)$ as the \emph{distribution vector} of $w$.
\end{definition}

In \sref{s:distro}, we derive restrictions on distribution vectors of symbols in PSCAs. These restrictions facilitate the computer searches that we use to exhaustively catalogue PSCA$(v, t, \lambda)$ for different sets of parameters. These searches and their results are described in \sref{s:comp}. All computational results reported in this paper were checked by both authors using independent computations.

Although there are still few parameters $(v, t)$ for which $g(v, t)$ has been determined exactly, there are known lower bounds for $g(v, t)$ for $t > 3$ and known upper and lower bounds for $g(v, 3)$, each of which are due to Yuster \cite{Yus19}. The first is proved by a matrix rank argument and states that if $t/2$ is a prime, then for $v \geq t$,
\begin{align*}
g(v, t) \geq \frac{\binom{v}{t/2} - \binom{v}{t/2-1}}{t!}.
\end{align*}
The upper bound for $g(v, 3)$ comes from a general construction for a PSCA with $(v, t) = (3^{n}, 3)$ for $n \geq 1$ that is built from an affine plane. The lower bound is proved using a similar matrix rank argument to the more general case above. Combining these bounds, we have that for an absolute constant $C$,
\begin{align*}
\frac{v}{6} \leq g(v, 3) \leq Cv(\log v)^{\log 7}.
\end{align*}

In \sref{s:grp}, we explore the relationship between groups and PSCAs and use this relationship to construct PSCAs with strengths 3 and 4, thereby improving upon the best known upper bounds for $g(v, 3)$ for $9 \leq v \leq 32$ and providing the first non-trivial bounds for $g(v, t)$ for $v \leq 24$ and $4\le t\le6$. \Tref{T:improv3} summarises the improvements to bounds on $g(v, 3)$ while \Tref{T:improv4} summarises the new results for $4\le t\le6$. The tables also incorporate the exact bounds shown in \sref{s:comp}.

\begin{table}[htb!]
\begin{center}
\begin{tabular}{c || c c }
$v$ & New bound & Old bound \\% & Base group\\
\hline
6--7 & 2\rlap{*} & 8 \\% & $C_{6}$\\
8 & 3\rlap{*} & 8 \\% & $C_{8}$, $D_{8}$, $E_{8}$\\
9 & 6 & 8 \\%& $C_{9}$\\
10--12 & 6 & 160 \\%& $C_{12}$\\
13--14 & 7 & 160 \\%& $E_{16}$
15--16 & 16 & 160 \\%& $E_{16}$
17--19 & 19 & 160\\
20--32 & 96 & 160 
\end{tabular}
\caption{\label{T:improv3}Improvements to known bounds on $g(v,3)$. An asterisk denotes that the new bound is exact.}
\end{center}
\end{table}

\begin{table}[htb!]
\begin{center}
\begin{tabular}{c  c || c }
$v$ & $t$ & New bound \\% & Base group\\
\hline
7 & 4 & 2\rlap{*}  \\% & $C_{6}$\\
8--12 & 4 & 18 \\% & $C_{8}$, $D_{8}$, $E_{8}$\\
13 & 4 & 234  \\%& $C_{9}$\\
14--21 & 4 & 5040\\
22& 4 &18\,480\\
23& 4 &425\,040\\
24& 4 &10\,200\,960\\
7--11 & 5 & 66\\
12 & 5 & 792\\
13--22 & 5 & 3696\\
23 & 5 & 85\,008\\
24 & 5 & 2\,040\,192\\
8--12 & 6 & 132\\
13--24 & 6 & 340\,032
\end{tabular}
\caption{\label{T:improv4}New bounds for $g(v,t)$ for $4\le t\le6$. An asterisk denotes that the new bound is exact.}
\end{center}
\end{table}

Independently, and using different methods, Na, Jedwab and Li \cite{Njl22} have also considered the problem of determining $g(v, t)$. They find that $g(6, 3) = g(7, 3) = g(7, 4) = 2$, while also demonstrating that $g(8, 3) \leq 3$, $g(9, 3) \leq 4$ and $g(7, 5) \leq 4$. They also show that for $(v, t) \in \{ (5, 3), (6, 3), (7, 3), (7, 4) \}$ a PSCA$(v, t, \lambda)$ exists if and only if $\lambda \geq 2$ while a PSCA$(8, 3, \lambda)$ exists for any $\lambda \geq 3$. Several of these results were originally reported in Na's Masters thesis \cite{Na21}; in particular, he reported that $g(7, 4) = 2$ before we computed our catalogue of PSCA$(7,4,2)$.

\section{Distribution vectors}\label{s:distro}

Recall the definition of distribution vectors given in \dref{d:distro}. The distribution vector of $w$ records the number of times a symbol $w$ appears in each column across a multiset of permutations. In this section, we will derive several restrictions on distribution vectors for symbols in PSCAs. We begin with the following lemma which limits the number of occurrences of each symbol in a PSCA across sets of consecutive columns.

\begin{lemma}\label{l:binomsum}
Let $X$ be a \textup{PSCA}$(v, t, \lambda)$ with $v \geq t \geq 2$ and $\lambda \geq 1$. Then, for $w \in [v]$ and for $0 \leq i \leq t-1$,
\[
\frac{\lambda (v-1)!}{(v - t)!} = \sum_{j = 0}^{v-1} d_{w}(j) \binom{j}{i}\binom{v-1 - j}{t-1 - i}.
\]
\end{lemma}

\begin{proof}
Let $X$ be a PSCA$(v, t, \lambda)$, let $i \in \{ 0, \dots, t-1 \}$ and let $w \in [v]$. Let $S = \{ s \in \sym_{v, t} : s(i) = w  \}$. Note that $\vert S \vert = (v-1)!/(v - t)!$ and each sequence in $S$ is covered by $\lambda$ permutations in $X$. For some $j \in[v]$, let $\pi \in X$ be one of the $d_{w}(j)$ permutations in $X$ such that $\pi(j) = w$ and consider how many sequences in $S$ are covered by $\pi$. There are $j$ symbols that appear before $w$ and $v-1 - j$ symbols that appear after $w$ in $\pi$. For every sequence in $S$, there are $i$ symbols appearing before $w$ and $t-1 - i$ symbols appearing after $w$. Hence, $\pi$ covers $\binom{j}{i}\binom{v-1 - j}{t-1 - i}$ sequences in $S$. The result follows.
\end{proof}

We can now use \lref{l:binomsum} to prove the following theorem.

\begin{theorem}\label{t:distthm}
Let $X$ be a \textup{PSCA}$(v, t, \lambda)$ with $v \geq t \geq 2$ and $\lambda \geq 1$. Then for $w \in [v]$ and for $1 \leq s < t$,
\begin{align}
\frac{1}{t! \lambda} \sum_{j = 0}^{v-1} j^{s} d_{w}(j)  = \frac{1}{v} \sum_{i = 0}^{v-1} i^{s}. \label{eq:distros}
\end{align}
\end{theorem}

\begin{proof}
Fix $s \in \{1,\ldots,t-1\}$ and $w \in [v]$. Let $\alpha_k(w) = \sum_{j = 0}^{v-1} j^{k}d_{w}(j)$ for $k \in \{1,\ldots,s\}$. Now, $X$ is a PSCA$(v,k+1,\lambda\binom{t}{k+1})$ and so, using \lref{l:binomsum} with $i = k$, we find that $\alpha_{k}(w)$ is a function of $\lambda$, $v$, $t$, $k$ and $\alpha_1(w),\dots,\alpha_{k-1}(w)$. So, proceeding by induction on $k$, we have that $\alpha_k(w)$ is independent of $w$ for each $k \in \{1,\ldots,s\}$. Thus,
\begin{align*}
v\alpha_{s}(w) = \sum_{w \in [v]}\alpha_{s}(w) = \sum_{\pi \in X}\sum_{i=0}^{v-1} i^{s} = t! \lambda \sum_{i=0}^{v-1} i^{s}
\end{align*}
and hence \eref{eq:distros} holds.
\end{proof}

We call a distribution vector that satisfies \eref{eq:distros} for parameters $(v, t, \lambda)$ a \emph{$(v, t, \lambda)$-feasible distribution}. We now prove some more facts about distribution vectors when certain restrictions on $v$ and $t$ are imposed. The following theorem demonstrates more stringent restrictions on the distribution vector whenever $t$ is an odd prime. Intuitively, it states that in a PSCA whose strength is an odd prime $p$ and whose order is not divisible by $p$, the number of occurrences of a symbol across all columns of a given equivalence class modulo $p$ is itself divisible by $p$.

\begin{theorem}\label{t:primedist}
Let $X$ be a \textup{PSCA}$(v, p, \lambda)$ with $p$ an odd prime and with $v \not\equiv 0 \bmod p$. For $w \in [v]$ and $0 \leq j \leq p-1$, let $y_{w}(j) =  \sum_{i \equiv j \bmod p} d_{w}(i)$. Then, $y_{w}(j) \equiv 0 \bmod p$.
\end{theorem}

\begin{proof}
Let $w \in [v]$. By \tref{t:distthm}, if $X$ is a PSCA$(v, p, \lambda)$, then for $1 \leq i \leq p-1$,
\begin{align*}
\sum_{j = 0}^{v-1} j^{i}d_{w}(j) = \frac{p! \lambda}{v} \sum_{j = 0}^{v-1} j^{i}.
\end{align*}
As $v$ is not divisible by $p$, the right hand side of the equation above must be divisible by $p$. Therefore, for $1 \leq i \leq p-1$,
\begin{align*}
  0\equiv\sum_{j = 0}^{v-1} j^{i}d_{w}(j)
  \equiv \sum_{j = 1}^{p-1} j^{i}y_{w}(j) \mod p. 
\end{align*}
This gives a system of $p-1$ linear equations in $p-1$ variables over the field $\mathbb{F}_{p}$. We can restate this system as
\begin{align*}
A
\begin{pmatrix}
y_{w}(1)\\
y_{w}(2)\\
\vdots \\
y_{w}(p-1)
\end{pmatrix}
= 
\begin{pmatrix}
0\\
0\\
\vdots \\
0
\end{pmatrix}
\end{align*}
where $A$ is a $(p-1) \times (p-1)$ matrix over $\mathbb{F}_{p}$ with $A_{i, j} = j^{i}$. Therefore, $A$ is a Vandermonde matrix and thus, $A$ is non-singular. Hence, the only solution to this system is $y_{w}(j) \equiv 0 \bmod p$ for all $j \in \{ 1, \dots, p-1\}$. As the number of permutations in $X$ is $p! \lambda \equiv 0 \bmod p$, it also follows that $y_{w}(0) \equiv 0 \bmod p$.
\end{proof}

When $v = t + 1$ and $t$ is even, the following lemma proves that all $(v, t, \lambda)$-feasible distribution vectors are palindromic.

\begin{lemma}\label{l:t+1sym}
Let $t$ be even and $d_{w}$ be a $(t + 1, t, \lambda)$-feasible distribution. Then, $d_{w}(t/2 - i) = d_{w}(t/2 + i)$ for $0 \leq i \leq t$.
\end{lemma}

\begin{proof}
With $v = t + 1$, \lref{l:binomsum} implies that
\begin{align*}
\lambda t! = (t - i)d_{w}(i) + (i + 1)d_{w}(i + 1)
\end{align*}
for $0 \leq i \leq t-1$. Therefore, for even $t$,
\begin{align*}
d_{w}\left( \frac{t}{2} - i  \right) &= \frac{\lambda t! - \left( \frac{t}{2} - i + 1 \right) d_{w}\left( \frac{t}{2} - i + 1 \right)}{\frac{t}{2} + i}\\
d_{w}\left( \frac{t}{2} + i  \right) &= \frac{\lambda t! - \left( \frac{t}{2} - i + 1 \right) d_{w}\left( \frac{t}{2} + i-1 \right)}{\frac{t}{2} + i}
\end{align*}
for $0 \leq i \leq t/2$. Then induction on $i$ shows that $d_{w}(t/2 - i) = d_{w}(t/2 + i)$ for $0 \leq i \leq t/2$. 
\end{proof}

We continue with the case where $v = t + 1$. The only possible PSCA$(t, t, \lambda)$ is a multiset containing $\lambda$ copies of $\sym_{v}$. Therefore, this is exactly the PSCA we obtain by deleting any symbol from a PSCA$(t + 1, t, \lambda)$ (throughout the paper, whenever we delete a symbol from a PSCA with ground set $[v]$, we assume the remaining symbols get relabelled to $[v-1]$ in an order preserving way). We use this fact to derive further restrictions for a PSCA$(t + 1, t, \lambda)$. Let $X$ be a PSCA$(v, t, \lambda)$. For $w \in [v]$ and $I \subseteq [v] \backslash \{ w \}$ with $\vert I \vert = i$, let $d_{I, w}$ be the number of permutations $\pi \in X$ such that $\pi(i) = w$ and $I = \{ \pi(j) : 0 \leq j \leq i-1 \}$.

\begin{theorem}\label{t:t+1dist}
Let $X$ be a \textup{PSCA}$(t + 1, t, 1)$, let $w \in [v]$, and let $0 \leq i \leq t$. Then for any $i$-subset $I \subseteq [v] \backslash \{ w \}$,
\begin{align*}
d_{I, w} = \frac{d_{w}(i)}{\binom{t}{i}}.
\end{align*}
\end{theorem}

\begin{proof}
Let $w \in [v]$. We proceed by induction on $i$. Note that the statement is trivially true for $i = 0$. Suppose the statement is true for some $i$ with $0 \leq i \leq t-1$ and consider the statement for $i + 1$. Let $I = \{ u_{1}, \dots, u_{i + 1} \} \subseteq [v] \backslash \{ w \}$ and let $J = I \backslash \{ u_{i + 1} \}$. The array formed by removing $u_{i + 1}$ from each permutation of $X$ is $\sym_{t}$. The number of permutations $\tau \in \sym_{t}$ for which $\tau(i) = w$ and $\{ \tau(j) : 0 \leq j \leq i-1 \} = J$ is $i!(t-1 - i)!$. These permutations exactly correspond to the permutations $\pi \in X$ such that either $\pi(i) = w$ and $\{ \pi(j) : 0 \leq j \leq i-1 \} = J$ or $\pi(i + 1) = w$ and $\{ \pi(j) : 0 \leq j \leq i \} = I$. Thus,
\begin{align*}
i!(t-1 - i)! = d_{J, w} + d_{I, w}.
\end{align*}
By the inductive hypothesis, $d_{J, w} = d_{w}(i)/\binom{t}{i}$. Therefore, for any two $(i + 1)$-subsets of $[v] \backslash \{ w \}$, $I$ and $I'$, $d_{w}(I) = d_{w}(I')$. The sum of $d_{I, w}$ as $I$ ranges over all $\binom{t}{i + 1}$ possible $(i + 1)$-subsets of $[v] \backslash \{ w \}$ must be $d_{w}(i + 1)$. Therefore $d_{I, w} = d_{w}(i + 1)/\binom{t}{i + 1}$, completing the induction.
\end{proof}

\begin{cor}\label{c:t+1divis}
  Let $X$ be a \textup{PSCA}$(t + 1, t, 1)$, and let $0 \leq i \leq t$. Then $d_{w}(i)$ is divisible by $\binom{t}{i}$ for all $w \in [v]$.
\end{cor}

In general, if it could be shown that there are no $(v, t, \lambda)$-feasible distributions for some choice of $v, t$ and $\lambda$, then it would imply that a PSCA$(v, t, \lambda)$ does not exist. However, it is possible to find $(v, t, \lambda)$-feasible distributions for infinitely many choices of $v, t,$ and $\lambda$. For example, if $t! \lambda$ is divisible by $v$, then a distribution vector with $d_{w}(i) = t! \lambda/v$ for $0 \leq i \leq v-1$ is $(v, t, \lambda)$-feasible. 

On the other hand, it is possible to use $(v, t, \lambda)$-feasible distributions to disprove the existence of a PSCA$(v, t, \lambda)$ even when such distributions do exist. For example, consider the (5, 3, 1)-feasible distributions. By \tref{t:primedist}, for such a distribution, $d_{w}(2) \in \{ 0, 3, 6 \}$. If $d_{w}(2) = 6$, then $d_{w} = (0, 0, 6, 0, 0)$ which violates \eref{eq:distros} for $s = 2$. Now suppose $d_{w}(2) = 3$. Again, by \tref{t:primedist}, $\{ d_{w}(0) + d_{w}(3), d_{w}(1) + d_{w}(4) \} = \{ 0, 3 \}$. As the reverse of a PSCA is also a PSCA, we can without loss of generality suppose $d_{w}(0) + d_{w}(3) = 0$. Then, for $s = 2$, \eref{eq:distros} reduces to $d_{w}(1) + 16d_{w}(4) = 24$. As $d_{w}(1)$ and $d_{w}(4)$ must be nonnegative integers that sum to 3, we find that this equation has no solutions. Therefore, in any (5, 3, 1)-feasible distribution, $d_{w}(2) = 0$. This means that if a PSCA(5, 3, 1) exists, then it would be impossible to place any symbol in column 2. This contradiction provides an alternate proof of the non-existence of a PSCA(5, 3, 1). See \cite{Math99} for an earlier proof.

In the proof of \tref{t:t+1dist}, we were able to enforce restrictions on a PSCA$(t + 1, t, \lambda)$ by considering the new array formed by deleting a symbol from this PSCA. We consider this kind of symbol deletion in a more general setting with the following theorem.

\begin{theorem}\label{t:compatdist}
  Let $\boldsymbol{d}_{w}=\big(d_{w}(0),\dots,d_{w}(v-1)\big)$ be the
  distribution vector for a symbol $w$ in $X$, a
  \textup{PSCA}$(v,t,\lambda)$. Let
  $\boldsymbol{d}'_{w}=\big(d'_{w}(0),\dots,d'_{w}(v-2)\big)$
  be the distribution vector of $w$
  in the \textup{PSCA} $X'$ obtained by deleting a symbol $w'\ne w$ from $X$.
  Then 
  \[
  \delta_k=\sum_{i=0}^k\big(d'_w(i)-d_w(i)\big)
  \]
  satisfies $0\le\delta_k\le d'_w(k)$ for $0\le k\le v-2$.
\end{theorem}

\begin{proof}
  Define $c_i=\big|\{\pi\in X:\pi^{-1}(w)=i<\pi^{-1}(w')\}\big|$
  and $c'_i=\big|\{\pi\in X:\pi^{-1}(w)=i>\pi^{-1}(w')\}\big|$
  for $0\le i\le v-1$. Then $c_i+c'_i=d_w(i)$ and $c_i+c'_{i+1}=d'_w(i)$.
  Now $c'_0=0$ and $c'_{i+1}-c'_i=d'_w(i)-c_i-c'_i=d'_w(i)-d_w(i)$.
  So it follows by induction on $i$ that $\delta_i=c'_{i+1}$ for
  $0\le i\le v-2$. The result then follows from the fact that
  $c_i\ge0$ and $c'_i\ge0$ for each $i$, by definition.
\end{proof}

We say that $\boldsymbol{d}_w$ and $\boldsymbol{d}'_w$ are \emph{compatible}
if they satisfy \tref{t:compatdist}. This test
can be used to eliminate some distributions from
consideration. If $\boldsymbol{d}_w$ is $(v, t, \lambda)$-feasible, it may be the case that there is no $(v-1, t, \lambda)$-feasible distribution
$\boldsymbol{d}'_w$ compatible with $\boldsymbol{d}_{w}$.
It may even happen that there is a compatible $\boldsymbol{d}'_w$, but that
all such candidates can themselves be ruled out because they are not
compatible with a $(v - 2, t, \lambda)$-feasible distribution, and so on.
A concrete example is that $(2,6,1,1,6,2)$ is a $(6,3,3)$-feasible distribution.
The only $(5,3,3)$-feasible distribution that it is compatible with is $(3,6,0,6,3)$. However $(3,6,0,6,3)$ is not compatible with any of the four $(4,3,3)$-feasible distributions, which are (3,9,0,6), (4,6,3,5), (5,3,6,4) and (6,0,9,3). Hence $(2,6,1,1,6,2)$ and $(3,6,0,6,3)$ can be eliminated from consideration.

\Tref{T:numdistro} records the number of $(v, t, 1)$-feasible distributions for different values of $v$ and $t$, as well as incorporating information about how many distributions cannot be ruled out using \tref{t:compatdist} in the manner just described.

\medskip

\begin{table}
\begin{center}

  \begin{tabular}{cc || c c c c c c cc}
%& \multicolumn{6}{c}{$v$}\\
&$t$ & $v=3$ & 4 & 5 & 6 & 7 & 8 & 9 & 10\\
\hline
$\lambda=1$&3 & 1/1 & 2/2 & 2/3 & 0/1 & 0/3 & 0/4 & 0/5 & 0/9\\
&4 & - & 1/1 & 3/3 & 6/6 & 8/13 & 19/30 & 36/57 & 61/119\\
&5 & - & - & 1/1 & 5/5 & 21/27 & 117/127 & 570/689 & 3359/3620 \\
%&6 & - & - & - & 1/1 & /13 & /177 & & \\[1ex]
\hline
$\lambda=2$&3 & 1/1 & 3/3 & 6/8 & 8/12 & 16/28 & 30/55 & 44/99 & 67/165\\
&4 & - & 1/1 & 5/5 & 17/17 & 59/74 & 261/291 & 1034/1128 & 3940/4235\\
&5&-&-&1/1&9/9&79/93&900/910&9267/9908&106859/107947\\
\hline
$\lambda=3$&3 & 1/1 & 4/4 & 11/14 & 32/37 & 84/99 & 224/252 & 547/609 & 1315/1409 \\
&4 & - & 1/1 & 7/7 & 35/35 & 195/221 & 1246/1296 & 7243/7341 & 38781/39486\\
&5& - & - & 1/1 & 13/13 & 179/199 & 2933/2951 & 46160/48150 & 790491/793171
  \end{tabular}
\caption{\label{T:numdistro}Number of feasible distributions. Each
  entry $r/s$ indicates that there are $s$ distributions that
  are $(v,k,\lambda)$-feasible, and that $r$ of these cannot be ruled out using
  \tref{t:compatdist}.}
\end{center}
\end{table}

\section{Exhaustive search algorithm}\label{s:comp}

We have seen in the previous section the relationship between a PSCA$(v, t, \lambda)$ and the smaller array that results from deleting a symbol from this PSCA. Specifically, we have seen that by deleting a symbol from a PSCA$(t + 1, t, \lambda)$, we are left with $\lambda$ copies of $\sym_{t}$. We can extend this argument to say that by deleting $v - t$ symbols from a PSCA$(v, t, \lambda)$, we obtain $\lambda$ copies of $\sym_{t}$. In this sense, every PSCA contains $\lambda$ copies of $\sym_{t}$. This relationship between smaller and larger PSCAs with the same strength and multiplicity allows for the design of an algorithm that can exhaustively search for a PSCA$(v, t, \lambda)$ by first cataloguing all possible PSCA$(v', t, \lambda)$ for $t \leq v' < v$. Such an algorithm is further aided by the results proved in the previous section. In order to catalogue all possible PSCAs for a particular choice of parameters, we must first establish a definition of isomorphism for PSCAs.

\begin{definition}\label{d:isomorph}
Two multisets of permutations, $X$ and $Y$, are \emph{isomorphic} if $Y$ can be obtained from $X$ by permuting the symbols and/or reversing every permutation.
\end{definition}

In searching for PSCA$(v, t, \lambda)$ for $v > t$, we employed two different methods. Both of these methods relied on a catalogue of isomorphism class representatives of PSCA$(v-1, t, \lambda)$. For each array in this catalogue, we tested every possible way of inserting a new symbol into each permutation of the array. In the first method, we assigned a $(v, t, \lambda)$-feasible distribution for this new symbol and found all possible PSCAs that can be formed when the new symbol obeys that distribution, before moving on to the next $(v, t, \lambda)$-feasible distribution. In the second method, we did not fix a distribution. Instead, we maintained a list of $(v, t, \lambda)$-feasible distributions that were consistent with the positions so far chosen for the new symbol. If that list ever became empty then we knew the current placements were unviable. Using these two search methods, we were able to independently count the number of isomorphism classes of PSCA$(v, t, \lambda)$ for different sets of parameters, as shown in \Tref{T:numPSCA}. In some cases it was not feasible to perform an exhaustive enumeration. In such cases, the number of PSCAs that we found before abandoning the search is given with a $+$ symbol indicating that the search was incomplete. In each such case we believe that the true number of PSCAs is much higher than the number that we quote.

In the cases when
$(v,t,\lambda)\in\{(5,3,1),(7,4,1),(7,5,1),(8,3,2),(8,4,2)\}$ our
enumeration was exhaustive, and demonstrated that no PSCA with these
parameters exists. For the first three of these parameter sets this was
already known, but the last two are new results. Our computations have
discovered several new values of the function $g$.

\begin{theorem}\label{t:gvalue}
$g(6,3)=g(7,3) = g(7, 4)=2$ and $g(8, 3) = 3$. Additionally, $g(8, 4) > 2$.
\end{theorem}

\begin{proof}
Given the nonexistence results just mentioned, it suffices to display a
PSCA$(7, 3, 2)$, a PSCA$(8, 3, 3)$ and a PSCA$(7, 4, 2)$:
\begin{center}
\begin{tabular}{c | c || c | c || c | c | c | c}
\multicolumn{2}{c}{PSCA(7, 3, 2)} & \multicolumn{2}{c}{PSCA(8, 3, 3)} & \multicolumn{4}{c}{PSCA(7, 4, 2)}\\
\hline
0123465& 0642315& 04712563& 05672341& 0123465& 0254163 & 0351264 & 0432165\\
1540362& 1634052& 06432157& 07351462& 0621435& 0634125 & 0651432 & 0652341\\
2405163& 2610543& 16547203& 17453026& 1045263& 1254063 & 1432560 & 1530264\\
3054261& 3625401& 17630245& 25476301& 1632045& 1635402 & 1640253 & 1652043\\
4312560& 4651230& 26751043& 27410365& 2045361& 2103564 & 2341560 & 2530164\\
5231064& 5603124& 31526074& 34675102& 2601534& 2635104 & 2643015 & 2645103\\
\multicolumn{1}{c}{}&& 37206154& 42351067& 3015462 & 3214065 & 3402561 & 3520461\\
\multicolumn{1}{c}{}&& 46051327& 50213476& 3604521 & 3610254 & 3614520 & 3625401\\
\multicolumn{1}{c}{}&& 53764201& 61234075& 4015362 & 4123065 & 4351062 & 4520163\\
\multicolumn{3}{c}{}&&4610352 & 4620351 & 4621530 & 4653012\\
\multicolumn{3}{c}{}&&5103462 & 5214360 & 5341260 & 5402361\\
\multicolumn{3}{c}{}&&5603214 & 5604123 & 5612340 & 5643210\\

\end{tabular}
\end{center}
\end{proof}

There are 260\,664 isomorphism classes of PSCA$(5,3,3)$. We took the one
which has the largest automorphism group and extended it in all
possible ways. Doing so produced 3072, 481\,765 and 51\,448 isomorphism
classes of PSCAs with parameters (6,3,3), (7,3,3) and (8,3,3)
respectively. However, none of these extended to a PSCA(9,3,3). We
also performed a search for all PSCA(6,3,3) in which every symbol
has distribution vector $(3,3,3,3,3,3)$. Using \tref{t:compatdist}, we
were able to find all (5,3,3)-feasible distributions that are
compatible with this uniform distribution and thus could determine the
PSCA(5,3,3) that could potentially extend to such a PSCA(6,3,3). From them
we found 1\,053\,700 PSCA(6,3,3) up to isomorphism. These arrays extend to
35\,872\,460 PSCA(7,3,3) and 1\,992\,709 PSCA(8,3,3) up to isomorphism.
Again, none of these
arrays extend to a PSCA(9,3,3). We also built some other PSCA(8,3,3)
via several other routes, but were unable to find a PSCA(9,3,3).

\begin{table}
  \begin{center}
\begin{tabular}{|ccc|cc|}
  \hline
  $t$&$\lambda$&$v$&PSCAs&groups\\
  \hline
  3&1&3&1&1\\
  3&1&4&1&0\\
  3&1&5&0&0\\
  \hline
  3&2&3&1&1\\
  3&2&4&12&1\\
  3&2&5&314&0\\
  3&2&6&1957&5\\
  3&2&7&146&0\\
  3&2&8&0&0\\
  \hline
  3&3&3&1&1\\
  3&3&4&37&0\\
  3&3&5&260\,664&0\\
  3&3&6&29\,100\,897+&0+\\
  3&3&7&14\,943\,804+&0+\\
  3&3&8&2\,111\,540+&0+\\
  \hline
  4&1&4&1&1\\
  4&1&5&4&0\\
  4&1&6&2&1\\
  4&1&7&0&0\\
  \hline
  4&2&4&1&1\\
  4&2&5&12\,351&0\\
  4&2&6&32\,507&2\\
  4&2&7&1826&0\\
  4&2&8&0&0\\
  \hline
  5&1&5&1&1\\
  5&1&6&3461&0\\
  5&1&7&0&0\\
  \hline
\end{tabular}
\caption{\label{T:numPSCA}Number of PSCAs generated by adding one symbol at a time.}
\end{center}
\end{table}

The last column  of \Tref{T:numPSCA} lists the number of isomorphism classes
in our catalogue which contain a PSCA for which
the corresponding set (ignoring multiplicity of repeated permutations)
of permutations forms a group. To test if a PSCA is isomorphic to
a group it suffices to permute the symbols to ensure that one permutation
(it does not matter which) is the identity, and then check that the
resulting set of permutations is closed under composition. PSCAs that
form groups will be studied further in the next section, which will
provide details of all of the groups included in \Tref{T:numPSCA} (except
the trivial cases when $v=t$).

\begin{table}
\begin{center}
\begin{tabular}{c c c || c c}
$v$ & $t$ & $\lambda$ & Realised Distributions & Compatible Distributions\\
\hline
3 & 3 & 1 & 1 & 1\\
4 & 3 & 1 & 2 & 2\\
\hline
3 & 3 & 2 & 1 & 1\\
4 & 3 & 2 & 3 & 3\\
5 & 3 & 2 & 6 & 6\\
6 & 3 & 2 & 4 & 8\\
7 & 3 & 2 & 2 & 16\\
\hline
3 & 3 & 3 & 1 & 1\\
4 & 3 & 3 & 4 & 4\\
5 & 3 & 3 & 11 & 11\\
6 & 3 & 3 & 26 & 32 \\
\hline
4 & 4 & 1 & 1 & 1\\
5 & 4 & 1 & 3 & 3\\
6 & 4 & 1 & 1 & 6\\
\hline
4 & 4 & 2 & 1 & 1\\
5 & 4 & 2 & 5 & 5\\
6 & 4 & 2 & 10 & 17\\
7 & 4 & 2 & 16 & 59\\
\hline
5 & 5 & 1 & 1 & 1\\
6 & 5 & 1 & 5 & 5\\
\hline
\end{tabular}
\caption{\label{T:realdistro} Number of realised distributions for different parameter sets.}
\end{center}
\end{table}

In \Tref{T:numdistro} we showed how many distributions might be
achieved by symbols in PSCAs. In the ``realised distributions'' column
of \Tref{T:realdistro} we show how many of these distributions are
actually realised within some PSCA.  For comparison, the column headed
``compatible distributions'' repeats the smaller of the two bounds we
had computed in \Tref{T:numdistro}. \Tref{T:realdistro} covers all cases
where we computed (non-empty) exhaustive catalogues. It also covers the
case $(v,t,\lambda)=(6,3,3)$, where we were able to rule out 6 distributions
with targeted searches, assisted by \tref{t:compatdist}. The 6 unrealised
distributions were (0,9,1,3,0,5), (2,6,0,4,3,3), (3,1,8,0,2,4) and
their reverses. The other 26 distributions from \Tref{T:numdistro}
appeared in our partial catalogue.

%3,3,4,0,6,2
%4,2,0,8,1,3
%5,0,3,1,9,0

\section{PSCAs from permutation groups}\label{s:grp}

In this section, we consider PSCAs which can be constructed from permutation groups. For permutations $f, g \in \sym_{v}$, the composition $f \circ g$ is the permutation $(f \circ g)(x) = f(g(x))$. For a subgroup $H \leq G$ and for $g \in G$, the right coset $Hg$ is the set $\{ hg : h \in H \}$ whereas the left coset $gH$ is the set $\{ gh : h \in H \}$. If $H$ is a subgroup of $\sym_{v}$, then the right coset $Hg$ permutes the columns of $H$ according to $g$ while the left coset $gH$ permutes the symbols of $H$ according to $g$. Throughout this section, for $s \in \sym_{v, t}$, we use the notation $s = (s_{0}, \dots, s_{t-1})$. Moreover, $G$ will always denote a group such that if $G$ has order $v$, then the elements of $G$ are $\{ 0 , \dots, v-1 \}$, and $\psi$ will denote an injective homomorphism $\psi : G \rightarrow \sym_{v}, g \mapsto \psi_{g}$. We can then consider the action of $G$ on $\sym_{v, t}$ where, if $s = (s_{0}, \dots, s_{t-1}) \in \sym_{v, t}$ and $g \in G$, then $gs = (\psi_{g}(s_{0}), \dots, \psi_{g}(s_{t-1}))$. 

Mathon and van Trung \cite{Math99} found that there are exactly two non-isomorphic PSCA$(6, 4, 1)$; one forms a group isomorphic to $\sym_{4}$, the other forms three cosets of a group isomorphic to $D_{8}$. While noting the connection between their PSCAs and groups, their search methods did not focus on building PSCAs from groups (the same is true of our work in \sref{s:comp}). However, several connections between PSCAs and groups have been formalised by Na, Jedwab and Li \cite{Njl22} and they found a number of examples of PSCAs based on groups. Note that the permutation composition convention used in \cite{Njl22} differs from the convention used here.

\begin{lemma}\label{l:orbitcoverage}
Let $G$ be a group, $\psi : G \rightarrow \sym_{v}$ be an injective homomorphism, $T$ be the image of $\psi$ and let $Th$ be a right coset of $T$. If $x$ and $y$ are sequences belonging to the same orbit under the action of $G$ on $\sym_{v, t}$, then $x$ and $y$ are covered by the same number of permutations in $Th$.
\end{lemma}

\begin{proof}
Let $x$ and $y$ be elements of $\sym_{v, t}$ that belong to the same orbit under the action of $G$. Then $gx = y$ for some $g \in G$. Let $0 \leq c_{0} < \dots < c_{t-1} \leq v-1$ and let $f \in Th$ such that $f(c_{i}) = x_{i}$ for $0 \leq i \leq t-1$. Then, $f$ covers $x$. Now consider $\psi_{g} \circ f$. As $f(c_{i}) = x_{i}$, $(\psi_{g}\circ f)(c_{i}) = \psi_{g}(x_{i})$ for $0 \leq i \leq t-1$. Therefore, $\psi_{g} \circ f$ covers $y$. Therefore, for every permutation in $Th$ that covers $x$, we can find a corresponding permutation that covers $y$. So, the number of permutations in $Th$ that cover $y$ is at least the number of permutations in $Th$ that cover $x$. By reversing the argument, and noting $x = g^{-1}y$, we find that the number of permutations in $Th$ that cover $x$ is at least the number of permutations in $Th$ that cover $y$. Thus, $x$ and $y$ are covered by the same number of permutations in $Th$.
\end{proof}

A consequence of \lref{l:orbitcoverage} is that in a right coset of a permutation group $\psi(G)$, we can determine the number of permutations covering each sequence in the orbit of a sequence $x$ under the action of $G$ on $\sym_{v, t}$ by simply finding the number of permutations in the coset that cover $x$. We will develop this point further in the context of transitive permutation groups with the following lemma.

\begin{lemma}\label{l:transitivegrp}
Let $G$ be a group, let $\psi : G \rightarrow \sym_{v}$ be an injective homomorphism such that the image, $T$, of $\psi$ is a transitive permutation group and let $X$ be an array constructed from right cosets of $T$. Furthermore, let $w \in [v]$, $0 \leq i \leq t-1$ and let $S = \{ s \in \sym_{v, t} : s_{i} = w  \}$. If every sequence in $S$ is covered by $\lambda$ permutations in $X$, then $X$ is a \textup{PSCA}$(v, t, \lambda)$.
\end{lemma}

\begin{proof}
Let $s \in \sym_{v, t}$. Then, as $T$ is transitive, there is a $g \in G$ such that $\psi_{g}(s_{i}) = w$. Therefore, the orbit of $s$ contains a sequence in $S$. As every orbit of the action of $G$ on $\sym_{v, t}$ contains a representative from $S$, then by \lref{l:orbitcoverage}, if every sequence in $S$ is covered by $\lambda$ permutations in $X$, then every sequence in $\sym_{v, t}$ is also covered by $\lambda$ permutations in $X$.
\end{proof}

\subsection*{Elementary abelian 2-groups}

Throughout this subsection, we use $E_{v}$ to denote an elementary abelian 2-group on the set $[v]$ with identity 0 and operation $\oplus$. Then for a group $E_{v}$, we fix $\psi : E_{v} \rightarrow \sym_{v}$ to be the homomorphism that maps $g \mapsto \psi_{g}$ where $\psi_{g}(x) = g \oplus x$. We then let $T$ be the image of $\psi$. Under this homomorphism, $gs = (g \oplus s_{0}, \dots, g \oplus s_{t-1})$ for $g \in E_{v}$ and $s \in \sym_{v, t}$. By construction, $T$ is a sharply transitive set of permutations, a fact that will be critical in what follows. We begin our analysis of elementary abelian 2-groups with an overview of PSCAs built from $E_{4}$. Within $\sym_{4}$, there are several subgroups isomorphic to $E_{4}$. However, the only one of these subgroups that is sharply transitive (and hence may be represented within $T$) is the following:
\begin{align*}
\begin{matrix}
0 & 1 & 2 & 3\\
1 & 0 & 3 & 2\\
2 & 3 & 0 & 1\\
3 & 2 & 1 & 0
\end{matrix}
\end{align*}

\begin{figure}
\begin{center}
\begin{tabular}{c  c  c}

\begin{tabular}{c c c c}
0 & 1 & 2 & 3\\
1 & 0 & 3 & 2\\
2 & 3 & 0 & 1\\
3 & 2 & 1 & 0
\end{tabular}

& 

\begin{tabular}{c c c c}
0 & 2 & 1 & 3\\
1 & 3 & 0 & 2\\
2 & 0 & 3 & 1\\
3 & 1 & 2 & 0
\end{tabular}

&

\begin{tabular}{c c c c}
0 & 3 & 1 & 2\\
1 & 2 & 0 & 3\\
2 & 1 & 3 & 0\\
3 & 0 & 2 & 1
\end{tabular}
\vspace{0.35cm}\\

\begin{tabular}{c c c c}
0 & 1 & 3 & 2\\
1 & 0 & 2 & 3\\
2 & 3 & 1 & 0\\
3 & 2 & 0 & 1
\end{tabular}

&

\begin{tabular}{c c c c}
0 & 2 & 3 & 1\\
1 & 3 & 2 & 0\\
2 & 0 & 1 & 3\\
3 & 1 & 0 & 2
\end{tabular}

&

\begin{tabular}{c c c c}
0 & 3 & 2 & 1\\
1 & 2 & 3 & 0\\
2 & 1 & 0 & 3\\
3 & 0 & 1 & 2
\end{tabular}
\vspace{0.35cm}\\

Type~A & Type~B & Type~C
\end{tabular}
\caption{\label{f:cosets}Three types of cosets of $E_4$}
\end{center}
\end{figure}

The cosets of this group within $\sym_{4}$ are shown in \fref{f:cosets}.
We refer to the cosets on the left as having Type~A coverage, the cosets in the middle as having Type~B coverage and the cosets on the right as having Type~C coverage. Cosets of the same type cover the same set of triples. Each coset covers 16 triples of $\sym_{4, 3}$ exactly once, leaving 8 triples uncovered. These uncovered triples are recorded in \Tref{T:uncovtrip}.

\begin{table}
\begin{center}
\begin{tabular}{c c c}
 Type~A & Type~B & Type~C\\
\hline
%Uncovered triples
 021 & 012 & 013\\
 031 & 032 & 023\\
 120 & 103 & 102\\
 130 & 123 & 132\\
 203 & 210 & 201\\
 213 & 230 & 231\\
 302 & 301 & 310\\
 312 & 321 & 320
\end{tabular}
\caption{\label{T:uncovtrip}Triples uncovered by cosets of Type~A, B and C.}
\end{center}
\end{table}

Observe that the sets of triples uncovered by Type~A, Type~B and Type~C cosets partition $\sym_{4, 3}$. Suppose $X$ is a PSCA$(4, 3, \lambda)$ which is built from a combination of cosets of our $E_{4}$ permutation group. As the number of permutations in $X$ is $6\lambda$, the total number of cosets that make up $X$ is $3\lambda/2$. Consider the triple 012. This triple is covered by Type~A and Type~C cosets but is not covered by Type~B cosets. Given that the number of permutations that cover 012 is $\lambda$, there must be $\lambda/2$ Type~B cosets. Similar arguments involving other triples (e.g. 021 and 013) demonstrate that $X$ must be built from $\lambda/2$ of each type of coset. Furthermore, because of the coverage properties of each coset type, any combination of $\lambda/2$ Type~A cosets, $\lambda/2$ Type~B cosets and $\lambda/2$ Type~C cosets will form a PSCA$(4, 3, \lambda)$. Therefore, an array built from a combination of cosets of $E_{4}$ will form a PSCA$(4, 3, \lambda)$ if and only if the array contains an equal number of each type of coset.  

We use this characterisation to aid us in our search for PSCAs from cosets of permutation representations of the elementary abelian 2-group of order $v$ with $v > 4$. Obviously these larger groups contain many subgroups isomorphic to $E_{4}$. As in the general case above, we isolate a subset of triples of $\sym_{v, 3}$ such that balanced coverage on these triples implies balanced coverage for every triple in $\sym_{v, 3}$.

\begin{lemma}\label{l:E4subgroups}
Let $\mathcal{H}$ be the set of order $4$ subgroups of $E_{v}$ and let $S$ be the set of triples defined by
\begin{align*}
S = \big\{ (s_{0}, s_{1}, s_{2}) \in \sym_{v, 3} : \{s_{0}, s_{1}, s_{2}\} \subset H \textup{ for some } H \in \mathcal{H}  \big\}
\end{align*}
Let $X$ be an array constructed from right cosets of $T$ in $\sym_{v}$. If every triple in $S$ is covered by $\lambda$ permutations in $X$, then $X$ is a \textup{PSCA}$(v, 3, \lambda)$.
\end{lemma}

\begin{proof}
First, we observe that if $\{ x, y, z \}$ is a 3-subset of an elementary abelian 2-group, then $\{ x, y, z, x \oplus y \oplus z \}$ is a coset of the order 4 subgroup $\{ 0, x \oplus y, x \oplus z, y \oplus z \}$. Furthermore, $x \oplus y \oplus z$ is the only element we can include with $\{ x, y, z \}$ in order to form an order 4 coset.

Let $(x, y, z) \in \sym_{v, 3}$. If $\psi_{x}$ acts on $(x, y, z)$, we obtain the triple $(0, x \oplus y, x \oplus z)$. As per the previous paragraph, $\{ 0, x \oplus y, x \oplus z \}$ forms a subset of an order 4 subgroup so $(0, x \oplus y, x \oplus z) \in S$. Hence, each orbit of $\sym_{v, 3}$ under the action of $E_{v}$ contains a triple from $S$. Therefore, by \lref{l:orbitcoverage}, if every triple in $S$ is covered by $\lambda$ permutations in $X$, then $X$ is a PSCA$(v, 3, \lambda)$.
\end{proof}

\begin{definition}\label{d:reduced}
Let $X \subseteq \sym_{v}$ be a multiset of permutations. For $W \subseteq [v]$, the \textit{reduced array of $X$ on $W$}, denoted by $X[W]$, is the array we obtain by removing every symbol of $[v] \backslash W$ from $X$.
\end{definition}

Let $Y$ be a right coset of $T$ in $\sym_{v}$, let $H$ be an order 4 subgroup of $E_{v}$ and consider the reduced array $Y[H]$. If we partition the rows of $Y[H]$ according to the cosets of $H$, then each part will form a coset of the sharply transitive $E_{4}$ permutation group. By taking $X$ to be a collection of right cosets of $T$, we can determine whether $X[H]$ forms a PSCA by analysing the coverage type of each coset of $E_{4}$ that appears in $X[H]$. As a result of \lref{l:E4subgroups}, if the reduced array $X[H]$ is a PSCA of strength 3 for each $H \in \mathcal{H}$, then $X$ will be a PSCA of strength 3.

\begin{lemma}\label{l:E4auto}
Let $f$ be an order $n$ automorphism of $E_{v}$ and let $X$ be the array
\begin{align*}
X = \bigcup_{i = 0}^{n-1} Tf^{i}.
\end{align*}
Let $H$ be an order $4$ subgroup of $E_{v}$. If the reduced array $X[H]$ is a \textup{PSCA}$(4, 3, \lambda)$, then $X[f^{i}(H)]$ will also be a \textup{PSCA}$(4, 3, \lambda)$ for $1 \leq i \leq n-1$.
\end{lemma}

\begin{proof}
First we show that $T = f^{-1}Tf$. Let $g \in E_{v}$. Then, we can consider $\psi_{g} \in T$ and the composition $f^{-1}\psi_{g}f$. Let $x \in E_{v}$. Then, $f^{-1}\psi_{g}f(x) = f^{-1}(g \oplus f(x))$. As $f$ is an automorphism of $E_{v}$, so too is $f^{-1}$. Hence $f^{-1}(g \oplus f(x)) = f^{-1}(g) \oplus x$. Therefore, $f^{-1}\psi_{g}f = \psi_{f^{-1}(g)}$ and hence, $f^{-1}Tf \subseteq T$. Now, $\psi_{g} = \psi_{f^{-1}(f(g))} = f^{-1}\psi_{f(g)}f$ by the above argument. So, $T \subseteq f^{-1}Tf$ and thus, $T = f^{-1}Tf$. Therefore, $fT = Tf$ and so we can consider $Tf$ as being an array in which the symbols of $T$ have been permuted according to $f$. As a result, the reduced array $T[H]$ is isomorphic to $Tf[f(H)]$. More generally, the reduced array $Tf^{i}[H]$ is isomorphic to $Tf^{i + 1}[f(H)]$ for $0 \leq i \leq n-1$. Moreover, the isomorphism in each case is the restriction of $f$ to $H$. Therefore, $X[H]$ is isomorphic to $X[f(H)]$. Applying this argument to the subgroups $f^{i}(H)$ and $f^{i + 1}(H)$ for $0 \leq i \leq n-1$, we find that $X[H]$ is isomorphic to $X[f^{i}(H)]$ for $1 \leq i \leq n-1$. Therefore, if $X[H]$ is a PSCA$(4, 3, \lambda)$, then so is $X[f^{i}(H)]$ for $1 \leq i \leq n-1$. 
\end{proof}

Essentially, \lref{l:E4subgroups} demonstrates that in a collection of right cosets of $T$, it suffices to check the coverage of triples whose elements form a subset of an order 4 subgroup $E_{v}$ in order to determine whether the cosets form a PSCA. When these cosets are related by an automorphism of $E_{v}$, we are able to further restrict what triples need to be checked by allowing us to consider only certain subgroups, depending on the automorphism $f$. In each case, the reduced array on any order 4 subgroup $H$ will form a collection of cosets of $E_{4}$ and so we can use the characterisation at the start of this section to determine whether these reduced arrays form PSCAs. Using these methods, we have been able to find PSCAs of orders 4, 8, 16 and 32 with strength 3. The following are examples of a PSCA(4, 3, 2) and a PSCA(8, 3, 4) (note that Na, Jedwab and Li \cite{Njl22} also found a PSCA$(8,3,4)$).

\begin{center}
\begin{tabular}{c || c | c}
\multicolumn{1}{c}{PSCA(4, 3, 2)} & \multicolumn{2}{c}{PSCA(8, 3, 4)}\\
\hline
0123& 01234567& 42671053\\
1032& 10543276& 53106742\\
2301& 25076143& 60435217\\
3210& 34701652& 71342506\\
0231& 43610725& 07245316\\
1320& 52167034& 16532407\\
2013& 67452301& 23061754\\
3102& 76325410& 32716045\\
0312& 06253471& 45607132\\
1203& 17524360& 54170623\\
2130& 24017635& 61423570\\
3021& 35760124& 70354261\\

\end{tabular}
\end{center}

The PSCA(4, 3, 2) forms a permutation group isomorphic to the alternating group $A_{4}$. The PSCA(8, 3, 4) forms a permutation group isomorphic to $A_{4} \times C_{2}$. We also have the following PSCAs of orders 16 and 32.

\begin{theorem}\label{t:E4gval}
$g(v, 3) \leq 16$ for $v \leq 16$ and $g(v, 3) \leq 96$ for $v \leq 32$.
\end{theorem}

\begin{proof}
To prove the first part of the theorem, we need only present a PSCA$(16, 3, 16)$. We let $G$ be the group isomorphic to $E_{16}$ generated by the permutations
\begin{center}
(0 1)(2 3)(4 5)(6 7)(8 9)(10 11)(12 13)(14 15),\\
(0 2)(1 3)(4 14)(5 15)(6 12)(7 13)(8 10)(9 11),\\
(0 4)(1 5)(2 14)(3 15)(6 10)(7 11)(8 12)(9 13),\\
(0 8)(1 9)(2 10)(3 11)(4 12)(5 13)(6 14)(7 15).
\end{center}
We then let $f =$(1 8 9)(2 4 15 11 5 7)(3 12 6 10 13 14). Then,
\begin{align*}
X = \bigcup_{i = 0}^{5} Gf^{i}
\end{align*}
forms a PSCA(16, 3, 16). The 96 permutations of this PSCA also form a group which can be generated by
\begin{center}
(1 8 9)(2 4 15 11 5 7)(3 12 6 10 13 14),\\
(0 4 7)(1 13 15)(2 3 10)(5 14 8)(6 9 12).
\end{center}

As a result of \lref{l:E4subgroups}, in order to check whether $X$ forms a PSCA, we need only check that the reduced arrays of $X$ corresponding to the 35 order 4 subgroups of $G$ form PSCA(4, 3, 16). As the cosets of $G$ from which $X$ is constructed are related by an automorphism, we can use \lref{l:E4auto} to further limit the number of reduced arrays of $X$ that we need to check in order to verify that $X$ is a PSCA. The orbits of the 35 order 4 subgroups of $G$ under $f$ are as follows.
\begin{align*}
&\big\{ \{ 0, 1, 2, 3 \}, \{0, 4, 8, 12 \}, \{ 0, 6, 9, 15 \}, \{0, 1, 10, 11\}, \{0, 5, 8, 13\}, \{0, 7, 9, 14\} \big\}\\
&\big\{ \{ 0, 1, 6, 7 \}, \{ 0, 2, 8, 10 \}, \{ 0, 4, 9, 13  \}, \{ 0, 1, 14, 15 \}, \{ 0, 3, 8, 11 \}, \{ 0, 5, 9, 12 \} \big\} \\
&\big\{ \{ 0, 2, 4, 14 \}, \{ 0, 3, 4, 15 \}, \{ 0, 11, 12, 15 \}, \{ 0, 5, 6, 11 \}, \{ 0, 5, 7, 10 \}, \{ 0, 2, 7, 13 \} \big\} \\
&\big\{ \{ 0, 2, 6, 12 \}, \{ 0, 4, 6, 10 \}, \{ 0, 10, 13, 15 \}, \{ 0, 11, 13, 14 \}, \{ 0, 3, 5, 14 \}, \{ 0, 3, 7, 12 \} \big\} \\
&\big\{ \{ 0, 1, 4, 5 \}, \{ 0, 7, 8, 15 \}, \{ 0, 2, 9, 11 \} \big\}\\
&\big\{ \{ 0, 1, 12, 13 \}, \{ 0, 6, 8, 14\}, \{0, 3, 9, 10 \} \big\}\\
&\big\{ \{ 0, 2, 5, 15 \}, \{ 0, 4, 7, 11 \} \big\} \\
&\big\{ \{ 0, 3, 6, 13 \}, \{ 0, 10, 12, 14 \} \big\}\\
&\big\{ \{ 0, 1, 8, 9 \} \big\}
\end{align*}
Hence, by \lref{l:E4auto}, we need only check the reduced array of one subgroup from each of these 9 orbits to verify that $X$ is a PSCA.

For the second part of the theorem, we present a PSCA$(32, 3, 96)$. We let $G_{32}$ be the group isomorphic to $E_{32}$ generated by the permutations
\begin{center}
(0 1)(2 3)(4 5)(6 7)(8 9)(10 11)(12 13)(14 15)(16 17)(18 19)(20 21)(22 23)(24 25)(26 27)(28 29)(30 31),\\
(0 2)(1 3)(4 28)(5 29)(6 30)(7 31)(8 10)(9 11)(12 20)(13 21)(14 22)(15 23)(16 18)(17 19)(24 26)(25 27),\\
(0 4)(1 5)(2 28)(3 29)(6 26)(7 27)(8 14)(9 15)(10 22)(11 23)(12 18)(13 19)(16 20)(17 21)(24 30)(25 31),\\
(0 8)(1 9)(2 10)(3 11)(4 14)(5 15)(6 12)(7 13)(16 24)(17 25)(18 26)(19 27)(20 30)(21 31)(22 28)(23 29),\\
(0 16)(1 17)(2 18)(3 19)(4 20)(5 21)(6 22)(7 23)(8 24)(9 25)(10 26)(11 27)(12 28)(13 29)(14 30)(15 31).
\end{center}
We then let $f_{1}$ be the following order 2 automorphism of $G_{32}$: 
\begin{center}
(2 8)(3 9)(4 6)(5 7)(12 28)(13 29)(14 30)(15 31)(18 24)(19 25)(20 22)(21 23).
\end{center}
We then let $G_{64} = G_{32} \cup G_{32} f_{1}$. Observe that $G_{64}$ also forms a group. Then, we let $f_{2}$ be the following order 3 automorphism of $G_{32}$:
\begin{center}
(2 12 24)(3 13 25)(4 6 10)(5 7 11)(8 18 28)(9 19 29)(20 22 26)(21 23 27).
\end{center}
We then let $G_{192} = G_{64} \cup G_{64}f_{2} \cup G_{64}f_{2}^{2}$. Again, 
$G_{192}$ forms a group. Finally, we let $f_{3}$ be the following order 3 automorphism of $G_{32}$:
\begin{center}
(1 16 17)(3 18 19)(5 20 21)(6 7 23)(9 24 25)(11 26 27)(12 13 29)(15 30 31).
\end{center}
Then $G_{192} \cup f_{3}G_{192} \cup f_{3}^{2}G_{192}$ is a PSCA$(32, 3, 96)$. Although this construction is not of the form described in \lref{l:E4auto}, it is a collection of right cosets of $G_{32}$. Therefore, we can use \lref{l:E4subgroups} to check that this array is indeed a PSCA.
\end{proof}

We remark that even though $G_{32}$, $G_{64}$ and $G_{192}$ are groups, the PSCA$(32, 3, 96)$ described in \tref{t:E4gval} is not a group. We also note that while $f_{3}$ is an automorphism of $G_{32}$, it is not an automorphism of $G_{192}$. As such, the shift to left cosets in the final step of the construction is significant as taking right cosets would not form a PSCA.

\begin{table}[htb!]
  \begin{center}
\begin{tabular}{| c | c | c |}
\hline
$(v, \lambda)$ & Group & Generators\\
\hline
(4,2) & $A_{4}$ & $\langle(1,2,3),(0,1,2)\rangle$ \\
\hline
\multirow{5}{*}{(6,2)} & $A_{4}$ & $\langle(0,5,4)(1,2,3),(0,5,1)(2,3,4) \rangle$ \\\cline{2-3}
& \multirow{4}{*}{$D_{12}$} & $\langle(0,5,4,2,1,3),(0,5)(1,2)(3,4)\rangle$\\
& & $\langle(0,4,5,2,1,3),(0,4)(1,2)(3,5)\rangle$\\
& & $\langle(0,3,1,5,4,2),(0,5)(1,3)(2,4)\rangle$\\
& & $\langle(0,3,1,4,5,2),(0,4)(1,3)(2,5)\rangle$\\
\hline
\multirow{6}{*}{(6,4)} & $C_{2} \times A_{4}$ & $\langle(0,5,1,2,3,4),(0,5,4)(1,2,3)\rangle$\\ \cline{2-3}
& \multirow{5}{*}{$\sym_{4}$} & $\langle(0,2)(1,3)(4,5),(0,4,5)(1,3,2)\rangle$\\
& & $\langle(0,2)(1,3)(4,5),(0,4,5)(1,2,3)\rangle$\\
& & $\langle(0,5)(1,3)(2,4),(0,2,4)(1,3,5)\rangle$\\
& &$\langle(1,3)(4,5),(0,1,5)(2,4,3)\rangle$\\
& &$\langle(1,3)(2,4),(0,1,4)(2,3,5)\rangle$\\
\hline
\multirow{22}{*}{(8,4)} & \multirow{2}{*}{$SL(2, 3)$} & $\langle(0,7,4,2)(1,5,3,6),(0,7,1)(2,3,4)\rangle$\\
& & $\langle(0,7,6,4)(1,3,2,5),(0,7,3)(4,5,6)\rangle$\\ \cline{2-3}
& \multirow{12}{*}{$\sym_{4}$} & $\langle(0,5,4,2)(1,7,3,6),(0,7,4)(1,3,2)\rangle$\\
& & $\langle(0,5,4,2)(1,6,3,7),(0,7,4)(1,3,5)\rangle$\\
& & $\langle(0,5,6,3)(1,4,7,2),(0,7,6)(2,4,5)\rangle$\\
& & $\langle(0,5,7,3)(1,4,6,2),(0,7,6)(2,5,4)\rangle$\\
& & $\langle(0,5,7,4)(1,3,6,2),(0,7,6)(2,5,3)\rangle$\\
& & $\langle(0,5,6,4)(1,3,7,2),(0,7,6)(2,3,5)\rangle$\\
& & $\langle(0,6,4,3)(1,7,5,2),(0,5,4)(2,7,6)\rangle$\\
& & $\langle(0,7,3,6)(1,2,4,5),(0,5,3)(1,4,7)\rangle$\\
& & $\langle(0,7,3,6)(1,5,4,2),(0,5,3)(1,4,6)\rangle$\\
& & $\langle(0,7,4,3)(1,6,5,2),(0,5,4)(2,6,7)\rangle$\\
& & $\langle(0,7,4,6)(1,3,5,2),(0,5,4)(2,3,7)\rangle$\\
& & $\langle(0,7,5,2)(1,6,4,3),(0,5,4)(2,3,6)\rangle$\\ \cline{2-3}
& \multirow{8}{*}{$C_{2} \times A_{4}$} & $\langle(1,4,7)(2,5,3),(0,1)(2,7)(3,6)(4,5)\rangle$\\
& &$\langle(1,4,6)(2,5,3),(0,1)(2,6)(3,7)(4,5)\rangle$\\
& &$\langle(1,2,5)(4,7,6),(0,1)(2,7)(3,6)(4,5)\rangle$\\
& &$\langle(1,2,5)(4,6,7),(0,1)(2,6)(3,7)(4,5)\rangle$\\
& &$\langle(1,6,7)(2,5,4),(0,2)(1,5)(3,6)(4,7)\rangle$\\
& &$\langle(1,6,7)(2,4,5),(0,2)(1,5)(3,7)(4,6)\rangle$\\
& &$\langle(1,6,7)(2,5,3),(0,2)(1,5)(3,7)(4,6)\rangle$\\
& &$\langle(1,6,7)(2,3,5),(0,2)(1,5)(3,6)(4,7)\rangle$\\
\hline
\multirow{9}{*}{(12,6)} & $C_{6} \times \sym_{3}$ & $\langle( 0,11,9,10, 1, 4)( 2, 7, 6, 3, 8, 5),( 0, 8,9, 2, 1, 6)( 3, 4, 5,11, 7,10)\rangle$\\ \cline{2-3}
& $\sym_{3} \times \sym_{3} $ & $\langle( 0,11,9,10, 1, 4)( 2, 5, 8, 3, 6, 7),( 0, 8,9, 2, 1, 6)( 3, 4, 7,10, 5,11)\rangle$\\ \cline{2-3}
&\multirow{7}{*}{$C_3\times A_4$}
 &$\langle( 2, 5,9)( 3, 6, 8)( 4,10,11),( 0, 2, 4)( 1,10, 8)( 3,9, 7)( 5,11, 6)\rangle$\\
&&$\langle( 2, 5,9)( 3, 6, 8)( 4,11,10),( 0, 2, 4)( 1,11, 8)( 3,9, 7)( 5,10, 6)\rangle$\\
&&$\langle( 2, 5,9)( 3, 6, 8)( 4, 7,11),( 0, 2, 4)( 1, 7, 8)( 3,9,10)( 5,11, 6)\rangle$\\
&&$\langle( 2, 5,9)( 3, 6, 8)( 4, 7,10),( 0, 2, 4)( 1, 7, 8)( 3,9,11)( 5,10, 6)\rangle$\\
&&$\langle( 2, 5,9)( 3, 6, 8)( 4,11, 7),( 0, 2, 4)( 1,11, 8)( 3,9,10)( 5, 7, 6)\rangle$\\
&&$\langle( 2, 5,9)( 3, 6, 8)( 4,10, 7),( 0, 2, 4)( 1,10, 8)( 3,9,11)( 5, 7, 6)\rangle$\\
&&$\langle( 1, 7,9)( 2, 6, 5)( 4,10, 8),( 0, 1, 8)( 2, 7,10)( 3, 5,9)( 4, 6,11)\rangle$\\
\hline
\multirow{2}{*}{(14,7)} & \multirow{2}{*}{$C_{7} \rtimes C_{6}$} & $\langle(1,4,7)(2,11,5)(3,9,13)(6,8,12),(0,2)(1,11)(3,10)(4,8)(5,13)(6,7)(9,12))\rangle$\\
&& $\langle(1,4,7)(2,11,5)(3,9,12)(6,8,13),(0,2)(1,11)(3,10)(4,8)(5,12)(6,7)(9,13))\rangle$\\ \cline{2-3}
\hline
\multirow{2}{*}{(16,16)} & \multirow{2}{*}{$(E_{16} \rtimes C_{2}) \rtimes C_{3}$} & $\langle (1, 8, 9)(2, 4, 15, 11, 5, 7)(3, 12, 6, 10, 13, 14), $\\
&& $\phantom{\langle}(0, 4, 7)(1, 13, 15)(2, 3, 10)(5, 14, 8)(6, 9, 12) \rangle$\\
\hline
\multirow{2}{*}{(19,19)} & \multirow{2}{*}{$C_{19} \rtimes C_{6}$} & $\langle (1,11,5,18,15,9)(2,7,12,8,3,13)(4,6,10,14,16,17),$\\
&&$ \phantom{\langle} (0,1,2,3,4,13,16,5,11,10,17,15,9,6,12,14,7,8,18) \rangle$\\
\hline
\end{tabular}
\caption{\label{T:grpPSCA3}Strength 3 PSCAs that are permutation groups}
\end{center}
\end{table}

\begin{table}[htb!]
\begin{center}
\begin{tabular}{| c | c | c | c |}
\hline
$(v, \lambda)$ & Group & Classes & Example\\
\hline
(6,1) &  $\sym_{4}$ & 1 & $\langle(1,3)(4,5),(0,1,4)(2,5,3)\rangle$\\
\hline
(6,2) & $C_{2} \times \sym_{4}$ & 1 &  $\langle (0, 5, 2, 1), (0, 1, 3, 2, 5, 4) \rangle$ \\
\hline
(7,7) & $PSL(3, 2)$ & 9 & $\langle (0, 2, 3, 4, 6, 5, 1), (0, 5, 4)(2, 6, 3) \rangle$\\
%& & $\langle (0, 2, 3, 4, 5, 6, 1), (0, 6, 4)(2, 5, 3) \rangle$\\
%& & $\langle (0, 2, 3, 5, 6, 4, 1), (0, 4, 5)(2, 6, 3) \rangle$\\
%& & $\langle (0, 2, 3, 6, 5, 4, 1), (0, 4, 6)(2, 5, 3) \rangle$\\
%& & $\langle (0, 2, 3, 5, 4, 6, 1), (0, 6, 5)(2, 4, 3) \rangle$\\
%& & $\langle (0, 2, 4, 3, 6, 5, 1), (0, 5, 3)(2, 6, 4) \rangle$\\
%& & $\langle (0, 2, 4, 3, 5, 6, 1), (0, 6, 3)(2, 5, 4) \rangle$\\
%& & $\langle (0, 2, 5, 3, 6, 4, 1), (0, 4, 3)(2, 6, 5) \rangle$\\
%& & $\langle (0, 2, 4, 5, 6, 3, 1), (0, 3, 5)(2, 6, 4] \rangle$\\
\hline
(8, 56) & $E_{8} \rtimes PSL(3,2)$ & 22 & $\langle (0, 7, 4, 2, 3, 1, 5), (0, 1, 3, 4)(2, 6, 7, 5)  \rangle $ \\
\hline
(9, 18) & $((E_{9} \rtimes Q_{8}) \rtimes C_{3} ) \rtimes C_{2}$ & 38 & $\langle (0, 8)(1, 3)(4, 5), (0, 3, 8)(1, 6, 4)(2, 7, 5) \rangle$\\
\hline
\multirow{2}{*}{(10, 30)} & $\sym_{6}$ & 102 & $\langle (0, 7)(2, 9)(3, 4), (0, 7, 5, 9, 1)(2, 3, 6, 8, 4) \rangle$\\
& $A_{6} . C_{2}$ & 51 & $\langle (0, 7, 3)(1, 2, 6)(4, 8, 5), (0, 5, 1, 6, 2, 4, 7, 3)(8, 9) \rangle$ \\
\hline
(12, 18) & $((E_{9} \rtimes Q_{8}) \rtimes C_{3} ) \rtimes C_{2}$ & 24 &  $\langle (2, 8)(3, 11)(6, 9)(7, 10), (0, 1, 9)(2, 4, 11)(3, 7, 5)(6, 10, 8) \rangle$\\ 
\hline
(13, 234) & $PSL(3, 3)$ & 130565 & $\langle (3, 9)(5, 7)(8, 10)(11, 12), (0, 1, 2, 3)(4, 11, 9, 8)(5, 12)(6, 10) \rangle$\\
\hline
\multirow{2}{*}{(21, 5040)} & \multirow{2}{*}{$PSL(3, 4) \rtimes \sym_{3}$} & \multirow{2}{*}{?} & $\langle (0, 16, 8, 9)(1, 4, 3, 20, 5, 13, 18, 19)(2, 6, 10, 7, 14, 17, 11, 12)\;\;$ \\
& & & $\;\;(0, 17, 7, 11, 10, 5, 4, 19)(1, 2, 16, 14, 15, 9, 12, 3)(8, 13, 18, 20) \rangle$\\
\hline
\end{tabular}
\caption{\label{T:grpPSCA4}Strength 4 PSCAs that are permutation groups }
\end{center}
\end{table}

\bigskip

Motivated by those PSCAs that we had earlier found which turned out to be permutation representations of groups, we decided to search for such objects directly. Fix $v,t$ and $\lambda$. We sought a representation in $\sym_v$ of some group of order $n=t!\lambda$. We began by deciding on positive integers $g_1,g_2$ and possibly $g_3$.  We then chose permutations of orders $g_1,g_2$ (and possibly $g_3$) and checked whether they generate a group of order $n$. For each group that we discovered in this way, we then tried to find a conjugate that was a PSCA. This was done by building up the PSCA one column at a time, backtracking whenever some $t$-sequence would be covered too many times. As the conjugate $h^{-1}Gh$ of a group $G$ is isomorphic in terms of sequence coverage to $Gh$, searching over all column permutations of $G$ for a PSCA is equivalent to searching over all conjugates of $G$. Since we checked all conjugates of each group that we found, we were free to insist that the generator of order $g_1$ that we chose was lexicographically maximal amongst all of its conjugates. In particular, this meant we only had to consider one choice for each possible cycle structure of that generator.  Note that this method did not prejudge which group it was going to build. Many non-isomorphic groups of order $n$ may have generators of the specified orders. For example, there are 15 groups of order 24, but they all have a generating set with $(g_1,g_2)\in\{(12,4),(12,2),(8,3),(6,4),(3,2)\}$ or $(g_1,g_2,g_3)=(6,6,2)$. Similarly, the 5 groups of order 18 all have a generating set with $(g_1,g_2)\in\{(9,2),(6,6)\}$ or $(g_1,g_2,g_3)=(3,3,2)$. Of course, groups will typically have many different generating sets with suitable orders, and hence will be built multiple times. But we could be confident that every group of order $n$ that has some representation in $\sym_v$ would be built, and thus that our catalogue of PSCAs that are groups is exhaustive for $v\le14$ and $n\le42$.

In an alternative computation, we used GAP \cite{gap} to generate representatives of conjugacy classes of subgroups of $\sym_{v}$ and used the backtracking process described above in order to search over each conjugacy class. We have also performed ad hoc computations on some doubly transitive permutation groups. Some of those groups had too many conjugates to search exhaustively, so we randomly sampled conjugates instead. Our results are recorded in two tables. The first, \Tref{T:grpPSCA3}, records permutation groups that are strength 3 PSCAs but not strength 4 PSCAs. The second table, \Tref{T:grpPSCA4}, records permutation groups that are strength 4 PSCAs but not strength 5 PSCAs. In \Tref{T:grpPSCA3}, a representative of each PSCA-isomorphism class of each group is presented. As a crosscheck, we note that these results agree with those presented in \Tref{T:numPSCA}, which were found by a completely separate method. For reasons of space, in \Tref{T:grpPSCA4} we do not list representatives of each PSCA-isomorphism class. Rather, we just give the number of such classes (or a ? when random sampling of conjugates was used instead of an exhaustive search).

We know of few permutation groups that are PSCAs of strength 5, other than symmetric and alternating groups. These necessarily include the 5-transitive Mathieu groups $M_{12}$ and $M_{24}$. Perhaps more interestingly, we also found that
\begin{equation}\label{e:M11}
  \langle(1,7)(2,8)(3,4)(6,9),(0,2,10,6)(3,7,5,8)\rangle
\end{equation}
is one of 108 presentations of the (4-transitive) Mathieu group $M_{11}$ in $\sym_{11}$ that form PSCAs of strength 5.
No subgroup of $\sym_{11}$ forms a 
PSCA of strength 4, other than those isomorphic to $M_{11}$, $A_{11}$ or $\sym_{11}$.
Similarly,
\begin{equation}\label{e:M12}
\langle (2, 11, 8, 6)(3, 10, 4, 5), (0, 1, 2, 3, 4, 5, 11, 6, 7, 10, 8), (0, 9)(1, 8)(2, 5)(3, 6)(4, 7)(10, 11) \rangle
\end{equation}
is one of 161 presentations of the (5-transitive) Mathieu group $M_{12}$ in $\sym_{12}$ that form PSCAs of strength 6.
The presentations of $M_{11}$ are conjugates of each other, and similarly for $M_{12}$.
If we let $r \in \sym_{v}$ be the reverse permutation, i.e. $r(i) = (v - 1 - i)$, then for a permutation group $G \leq \sym_{v}$, we will find that $G$ and $rGr$ are isomorphic in terms of sequence coverage. Hence it is plausible that we may find presentations of the same group that are isomorphic as PSCAs. Indeed, this is the case for $M_{11}$ where the 108 presentations that form PSCAs of strength 5 can be reduced to 54 isomorphism classes. Meanwhile, the presentation of $M_{12}$ given in \eref{e:M12} is the only one of the 161 strength 6 PSCAs for which conjugation by $r$ leaves the underlying set of permutations unchanged. Thus, these 161 presentations that form PSCAs of strength 6 reduce to 81 isomorphism classes.

For the larger Mathieu groups we were unable to do exhaustive computations and again relied on random sampling. We found that 
\begin{equation}\label{e:M22}
  \begin{aligned}
  \langle&(0,1,20,4,2)(3,8,9,12,13)(5,16,10,11,18)(6,7,15,19,14),\\
  &(0,13,16,5,10)(1,14,19,4,2)(3,18,7,12,15)(9,21,11,20,17)\rangle
  \end{aligned}
\end{equation}
is a presentation of the (3-transitive) Mathieu group $M_{22}$ in $\sym_{22}$ that forms a PSCA of strength 5. Also
\begin{equation}\label{e:M24}
  \begin{aligned}
  \langle&(0,1,2,3,4,5,6,7,8,9,10,11,12,13,14,15,16,17,18,19,20,21,22),\\
 &(0,23)(1,22)(2,11)(3,15)(4,17)(5,9)(6,19)(7,13)(8,20)(10,16)(12,21)(14,18),\\
 &(2,16,9,6,8)(3,12,13,18,4)(7,17,10,11,22)(14,19,21,20,15)
  \rangle
  \end{aligned}
\end{equation}
is a presentation of the (5-transitive) Mathieu group $M_{24}$ in $\sym_{24}$ that forms a PSCA of strength 6. Its point stabilisers provide PSCAs of strength 5 in $\sym_{23}$ that are presentations of $M_{23}$.

%% For Gap, this M22 is
%M22:=Group((1,2,21,5,3)(4,9,10,13,14)(6,17,11,12,19)(7,8,16,20,15),(1,14,17,6,11)(2,15,20,5,3)(4,19,8,13,16)(10,22,12,21,18));

%%[[IMW note to self: Currently have 17116 isomorphism classes of the PSCA(16,3,16)
%% see ~/project/ian/PSCA/conjugate16/OurPSCA16-96.variants]]

\Tref{T:grpPSCA3} also includes the PSCA$(16, 3, 16)$ found earlier in
the section, and a PSCA$(19,3,19)$. Exhaustive searches were not
undertaken for either of these parameter sets. However, a partial
search found 17116 and 232 isomorphism classes, respectively, of
PSCA$(16,3,16)$ and PSCA$(19,3,19)$ that are conjugate to the examples
given in the table. Note that since isomorphism includes the option to
freely permute symbols, the only material effect of conjugation in
this context is to permute the columns of a PSCA.

A striking feature of results summarised in
\Tref{T:grpPSCA3} and \Tref{T:grpPSCA4} is that there are a number of cases of
non-isomorphic PSCAs being produced by similar sets of generators. For example,
starting from the PSCA$(6, 4, 1)$, if we conjugate the generating set by
the transposition $(2,5)$ we reach a PSCA$(6,3,4)$. A similar thing happens
if we use the transposition $(4,5)$. Conjugating the generating
set by a transposition has the effect of interchanging two columns of the
PSCA (and then exchanging two symbols to once again achieve the property
of having one row equal to the identity permutation).

Summarising our bounds on $g(v,t)$ derived from group presentations, we have:

\begin{theorem}\label{t:str4gval}
\mbox{ }
\begin{itemize}
\item For $v \le 11$, we have $g(v, 5) \le 66$,
\item For $v \le 12$, we have $g(v, 4) \le 18$, $g(v, 5) \le 792$ and $g(v, 6) \leq 132$,
\item For $v \le 13$, we have $g(v, 4) \le 234$,
\item For $v \le 21$, we have $g(v, 4) \le 5040$, and
\item For $v \le 22$, we have $g(v, 5) \le 3696$ and hence $g(v, 4) \le 18\,480$.
\item For $v \le 23$, we have $g(v, 5) \le 85\,008$ and hence $g(v, 4) \le 425\,040$.
\item For $v \le 24$, we have $g(v, 6) \le 340\,032$ and hence $g(v, 5) \le 2\,040\,192$ and $g(v,4)\le10\,200\,960$.
\end{itemize}
\end{theorem}

\begin{proof}
  Examples of a PSCA$(12, 4, 18)$, a PSCA$(13, 4, 234)$ and a PSCA$(21, 4, 5040)$ are given in \Tref{T:grpPSCA4}.
  Also, we gave a PSCA$(11,5,66)$ in
  \eref{e:M11}, a PSCA$(12,6,132)$ in \eref{e:M12}, a PSCA$(22,5,3696)$ in \eref{e:M22} and a PSCA$(24,6,340\,032)$ in \eref{e:M24}, from which we derived a PSCA$(23,5,85\,008)$.
\end{proof}

\subsection*{Acknowledgements}
The authors are grateful to Jingzhou Na, Jonathan Jedwab and Shuxing Li
for sharing the results of their ongoing investigation \cite{Na21}, \cite{Njl22}, 
which has paralleled our own.
We are also very grateful to Daniel Horsley who has been very generous with
his time and advice. 
This research was supported by the Monash eResearch Centre
through the use of the MonARCH
HPC Cluster. Computations in \sref{s:grp} were facilitated by GAP software
\cite{gap}.

%%%  Squashing the bibliography together a bit
 
  \let\oldthebibliography=\thebibliography
  \let\endoldthebibliography=\endthebibliography
  \renewenvironment{thebibliography}[1]{%
    \begin{oldthebibliography}{#1}%
      \setlength{\parskip}{0.25ex}%
      \setlength{\itemsep}{0.25ex}%
  }%
  {%
    \end{oldthebibliography}%
  }

\end{document}